\documentclass{amsart}
\usepackage{amssymb,amscd}

 \newtheorem{thm}{Theorem}[section]
 \newtheorem{cor}[thm]{Corollary}
 \newtheorem{lemma}[thm]{Lemma}
 \newtheorem{prop}[thm]{Proposition}
 \theoremstyle{definition}
 \newtheorem{defn}[thm]{Definition}
 \newtheorem{exmp}[thm]{Example}
 \theoremstyle{remark}
 \newtheorem{rem}[thm]{Remark}

 \numberwithin{equation}{subsection}
 \newtheorem{ack}{Acknowledgment}
 
\newcommand{\UU}{\text{$\mathcal{U}$}}
\newcommand{\FF}{\text{$\mathcal{F}$}}
\newcommand{\GG}{\text{$\mathcal{G}$}}
\newcommand{\VV}{\text{$\mathcal{V}$}}

\newcommand{\HH}{\text{$\mathcal{H}$}}

\newcommand{\LB}{\text{$\Lambda$}}
\newcommand{\sg}{\text{$\sigma$}}
\newcommand{\second}{{\text{\rm II}}}

\newcommand{\infim}{\operatorname{inf}}

\newcommand{\Cat}{\operatorname{Cat}}
\newcommand{\dm}{\operatorname{dim}}
\newcommand{\str}{\operatorname{sat}}

\newcommand{\intr}{\operatorname{int}}

\newcommand{\Image}{\operatorname{Im}}

\newcommand{\supp}{\operatorname{supp}}
\newcommand{\dom}{\operatorname{dom}}

\newcommand{\length}{\operatorname{length}}
\newcommand{\Homeo}{\operatorname{Homeo}}

\newcommand{\growth}{\operatorname{growth}}

        \newcommand{\field}[1]{\text{$\mathbb{#1}$}}
        \newcommand{\N}{\field{N}}

        \newcommand{\R}{\field{R}}

\newdimen\theight
\def\TeXref#1{%
             \leavevmode\vadjust{\setbox0=\hbox{{\tt
                     \quad\quad  {\small \textrm #1}}}%
             \theight=\ht0
             \advance\theight by \lineskip
             \kern -\theight \vbox to
             \theight{\rightline{\rlap{\box0}}%
             \vss}%
             }}%

\begin{document}

\title{Secondary LS category  of measured laminations}

\author{Carlos Meni\~no Cot\'on}

\address{Departamento de Xeometr\'{\i}a e Topolox\'{\i}a\\
         Facultade de Matem\'aticas\\
         Universidade de Santiago de Compostela\\
         15782 Santiago de Compostela}

\email{carlos.meninho@gmail.com}

\thanks{Supported by MICINN (Spain): FPU program and Grant MTM2008-02640}

%%% ----------------------------------------------------------------------

\begin{abstract}
In \cite{Menino,Menino1,Menino2}, we have introduced a version of the tangential LS category \cite{Macias} for foliated spaces depending on a transverse invariant measure, called the measured category. Unfortunately, the measured category vanishes easily. When it is zero, the rate of convergence to zero of the quantity involved in the definition, by taking arbitrarily large homotopies, gives a new invariant, called the secondary measured category. Several versions of classical results are proved for the secondary measured category. It is also shown that the secondary measured category is a transverse invariant related to the growth of pseudogroups.
\end{abstract}

%%% ----------------------------------------------------------------------
\maketitle
%%% ----------------------------------------------------------------------

\section{Introduction}
The LS category is a homotopy invariant given by the minimum number of
open subsets contractible within a topological space needed to
cover it. It was introduced by L.~Lusternik and L.~Schnirelmann in 1930's in the setting of variational problems. Many variations of this invariant has been given. In
particular, E.~Mac\'ias and H.~Colman introduced a tangential version for
foliations, where they use leafwise contractions to transversals
\cite{HellenColman, Macias}. In \cite{Menino,Menino1,Menino2}, we defined another
version of the tangential category called measured LS category. We show that this invariant is zero in many dynamically interesting examples. So we are going to introduce dynamics in the definition of the measured category giving rise to the secondary category.

We begin the paper with a brief introduction on laminations, tangential category and measured category . Also, we show what is the reason underlying the nullity of the measured category in interesting dynamics.

 We continue introducing the secondary version of the measured LS category for laminations on compact spaces. For each positive integer $n$, the $n$-measured LS category is defined like the measured LS category by considering only deformations of (leafwise) length $\le n$, where the length of a deformation is the maximun length of its leafwise paths, defined by using chains of plaques of a fixed foliated atlas, or by using a Riemannian metric in the case of smooth foliations. The measured LS category is zero just when the $n$-measured LS category tends to zero as $n\to\infty$. In this case, the secondary category is defined as the growth type of the sequence of inverses of $n$-measured LS categories. Secondary category can be defined also at the level of pseudogroups and we obtain that it is in fact a transverse invariant (considering good generators on the pseudogroup). 

Finally, we study its relation with the growth of groups and pseudogroups (see {\em e.g.\/} \cite{Walczak}), obtaining a relation between these growth types and an equality in the case of free suspensions by Rohlin groups \cite{Series}. This final result allows to compute our invariant in simple examples and give us a criteria to find groups that do not satisfy the Rohlin condition.

\section{Definition of the \LB-category}

We refer to \cite{Candel-Conlon} for the basic notion and definitions about laminations, foliated chart, foliated atlas, holonomy pseudogroup and transverse invariant measure. They are recalled here in order to fix notations.

A {\em Polish space\/} is a second countable and complete metrizable space. Let $X$ be a Polish space. A {\em foliated chart\/} in $X$ is a pair $(U,\varphi)$ such that $U$ is an open subset of $X$ and $\varphi:U\to B^n\times S$ is an homeomorphism, where $B^n$ is an open ball of $\R^n$ and $S$ is a Polish space. The sets $B^n\times\{\ast\}$ are called the {\em plaques\/} of the chart, and the sets of the form $\varphi^{-1}(\{\ast\}\times S)$ are called the {\em associated transversals\/}. The map $U\to S$ is called the projection associated to $(U,\varphi)$. A {\em foliated atlas\/} is a family of foliated charts, $\{(U_i,\varphi_i)\}_{i\in I}$, that covers $X$ and the change of coordinates between the charts preserves the plaques; i.e., they are locally of the form $\varphi_i\circ\varphi_j^{-1}(x,s)=(f_{ij}(x,s),g_{ij}(s))$; these maps $g_{ij}$ form the holonomy cocycle associated to the foliated atlas. A {\em lamination\/} $\FF$ on $X$ is a maximal foliated atlas satisfying the above hypothesis. The plaques of the foliated charts of a maximal foliated atlas form a base of a finer topology of $X$, called the {\em leaf topology\/}. The connected components of the leaf topology are called the {\em leaves\/} of the foliation. The {\em dimension\/} of the lamination is the dimension of the plaques when all of them are open sets of the same Euclidean space.

A foliated atlas, $\UU=\{(U_i,\varphi_i)\}_{i\in \N}$, is called {\em regular\/} if it satisfies the following properties:
\begin{enumerate}
  \item [(a)] It is locally finite.
  \item [(b)] If a plaque $P$ of any foliated chart $(U_i,\varphi_i)$ meets another foliated chart $(U_j,\varphi_j)$, then $P\cap U_j$ is contained in only one plaque of $(U_j,\varphi_j)$.
  \item [(c)] If $U_i\cap U_j\neq \emptyset$, then there exists a foliated chart $(V,\psi)$ such that $U_i\cup U_j\subset V$, $\varphi_i=\psi|_{U_i}$ and $\varphi_j=\psi|_{U_j}$.
\end{enumerate}
Any topological lamination admits a regular foliated atlas, also we can assume that all the charts are locally compact. For a regular foliated atlas $\UU=\{(U_i,\varphi_i)\}_{i\in \N}$ with $\varphi_i:U_i\to B_{i,n}\times S_i$, the maps $g_{ij}$ generate a pseudogroup on $\bigsqcup_iS_i$. Holonomy pseudogroups defined by different foliated atlases are equivalent in the sense of \cite{Haefliger}, and the corresponding equivalence class is called the {\em holonomy pseudogroup\/} of the lamination; it contains all the information about its transverse dynamics. A {\em transversal\/} $T$ is a topological set of $X$ so that, for every foliated chart, $(U,\varphi)$, the corresponding projection restricts to a local homeomorphism $U\cap T\to S$. A transversal is said to be {\em complete\/} if it meets every leaf. On any complete transversal, there is a representative of the holonomy psudogroup which is equivalent to the representative defined by any foliated atlas via the projection maps defined by its charts. A {\em transverse invariant measure\/} is a measure on the space of any representative of the holonomy psudogroup invariant by its local transformations; in particular, it can be given as a measure on a complete transversal invariant by the corresponding representative of the holonomy pseudogroup. A {\em transverse set\/} is a measurable set that meets each leaf in a countable set. Any transverse invariant measure can be extended to all transverse sets so that they become invariant by measurable transformations that keep each point in the same leaf \cite{Connes}.

Observe that a measure is regular if the measure of a set $B$ is the infimum of the measures of the open sets containing it and the supremum of the measures of compact sets contained in $B$. Note that measures in the conditions of the Riesz representation theorem are regular, so it is not a very restrictive condition on measures.

\begin{rem}
Observe that the tangential model of the charts could be changed in order to give a more general notion of lamination: instead of taking open balls of $\R^n$ as $B^n$, we could take connected and locally contractible Polish spaces or separable Hilbert spaces. Also, it is possible to define the notion of $C^r$ foliated structure by assuming that the tangential part of changes of coordinates are $C^r$, with the leafwise derivatives of order $\le r$ depending continuously on the transverse coordinates. We can speak about regular atlases in Hilbert laminations but we cannot assume that its foliated charts are locally compact.
\end{rem}

Let us recall the definition of tangential category \cite{HellenColman,Macias}.
A lamination $(X,\FF)$ induces a foliated
measurable structure $\FF_U$ in each open set $U$. The space $U\times\R$ admits an
obvious foliated structure $\FF_{U\times\R}$ whose leaves are
products of leaves of $\FF_U$ and $\R$. Let $(Y,\GG)$ be another measurable lamination. A foliated map $H:\FF_{U\times\R}\to \GG$ is called a {\em tangential
homotopy\/}, and it is said that the maps $H(\cdot,0)$ and $H(\cdot,1)$ are {\em tangentially homotopic\/}. We use the term {\em
tangential deformation\/} when $\GG=\FF$ and $H( -
,0)$ is the inclusion map of $U$. A deformation such that $H( - ,1)$ is constant on the leaves of
$\FF_U$ is called a {\em tangential contraction\/} or an \FF-{\em
contraction\/}; in this case, $U$ is called a {\em tangentially categorical\/} or \FF-{\em categorical\/} open set. The {\em tangential category\/} is the lowest number of
categorical open sets that cover the measurable lamination. On one leaf foliations, this definition agrees with the classical category. The category of \FF\ is
denoted by $\Cat(\FF)$. It is clear that it is a tangential homotopy invariant.

The following result is about the structure of a tangentially categorical open set.

\begin{lemma}[Singhof-Vogt~\cite{Vogt-Singhof}]\label{l:vogt0}
Let $\FF$ be a foliation of dimension $m$ and codimension $n$ on a
manifold $M$, let $U$ be an \FF-categorical open set, let $x\in
U$, let $D\subset U$ be a transverse manifold of dimension $n$, and
suppose that $x$ belongs to the interior of $D$. Then there exists
a neighborhood $E$ of $x$ in $D$ such that any leaf of $\FF_{U}$
meets $E$ in at most one point.
\end{lemma}

Therefore, the final step of a tangential contraction is a countable union of local transversals of the foliation. Therefore it can be measured by a transverse invariant measure.

Let $\LB$ be a transverse invariant measure for
\FF\ and let $U$ be a tangentially categorical open set. Define
$$
  \tau_\LB(U)=\infim\{\,\LB(H(U\times\{1\})\mid H\ \text{is a tangential contraction of}\ U\,\}\;.
$$
Then the \LB-{\em category\/} of $(\FF,\LB)$ is defined as
$$\Cat(\FF,\LB)=\infim_{\UU}\sum_{U\in\UU}\tau_\LB(U)\;,$$
where \UU\ runs in the coverings of $X$ by \FF-categorical open sets. Observe that, for one leaf laminations with the transverse invariant measure given by the counting measure, the measured category agrees with the classical LS category.

\section{Nullity of the \LB-category}

\LB-category inherits some of the properties of the tangential and classical LS category , like homotopy invariance or a relation with the measure of the leafwise critical points of differentiable maps \cite{Menino}. But we are interesting in the behavior of the \LB-category in interesting dynamics like minimal or ergodic laminations. Minimal sets are nuclear in the theory of dynamical systems and foliations. We shall show that, under general conditions, \LB-category is zero.

\begin{defn}
Let $(X,\FF,\LB)$ be a lamination with a transverse
invariant measure. A {\em null-transverse\/} set is a transverse set
$B$ such that $\LB(B)=0$.
\end{defn}

The following propositions are elementary.

\begin{prop}\label{p:nullset}
Let $(X,\FF,\LB)$ be a lamination with a transverse
invariant measure, and let $B$ be a null-transverse set. Then
$\Cat(\FF,\LB)$ can be computed by using only coverings of $X\setminus \str(B)$ by open sets in $X$, where $\str(B)$ denotes the saturation of $B$ in \FF. If $B$ is saturated and closed, then $\Cat(\FF,\LB)=\Cat(\FF_{X\setminus B},\LB_{X\setminus B})$.
\end{prop}

Now, we explain an important lemma that allows us to cut tangentially categorical open sets into small ones, in order to compute the \LB-category when the measure is finite on compact sets. In this process, a null transverse set is generated, but we know that this kind of set can be removed in the computation of \LB-category.

\begin{lemma}\label{l:disgregation}
Let $(\FF,\LB)$ be a lamination with a transverse invariant measure that is finite on compact sets. Let $U$ be a tangentially categorical open set and let $T$ be a complete transversal. Then there exist a partition $\{F,U_n\}_{n\in\N}$  of $U$ such that $F$ is a closed and null transverse set, each $U_n$ is open, and there exists a tangential homotopy $H:\bigcup_n U_n\times \R\to \FF$ such that $H(\bigcup_n U_n\times\{1\})\subset T$.
\end{lemma}

\begin{proof}
Let $G$ be a tangential contraction of $U$. By Lemma~\ref{l:vogt0}, $H(U\times\{1\})=S$ is a countable union of transversals. Therefore there exist a countable covering of $S$, $S=\bigcup_i S_i$, by transversals such that, for each $i\in\N$, there exists a holonomy map $h_i$ with domain $S_i$  and image contained in $T$ (since $T$ is a complete transversal). Each $h_i$ induce a tangential deformation $H^i:S_i\times \R\to\FF$ such that $H^i(S_i\times\{1\})=h_i(S_i)$. By Lemma~\ref{l:measure finite in compact sets}, we can suppose that $\LB(\partial S_i)=0$ for $i\in\N$.

Now, take $T_1=S_1$ and define recursively $T_n=\intr(V_n\setminus\bigcup_{i=1}^{n-1}S_i)$ and $K=U\setminus\bigcup_n T_n$. Of course, $K$ is closed in $S$ and it is null transverse since $K\subset\bigcup_i\partial S_i$. Finally, we obtain the partition by taking $U_n=H(-,1)^{-1}(T_n)$  and $F=H(-,0)^{-1}(K)$. The tangential contraction $H$ for $\bigcup_n U_n$ is defined as follows:
$$
H(x,t)=
\begin{cases}
G(x,2t) & \text{if $t\leq \frac{1}{2}$}\\
H^i(G(x,1),2t-1) & \text{if $t\geq\frac{1}{2}$ and $x\in U_i$\;.}\qed
\end{cases}
$$
\renewcommand{\qed}{}
\end{proof}

The way to define deformations used in the above prove will be useful in other sections of this work. This idea is specified by defining the following homotopy operation.

\begin{defn}\label{d:operation homotopy}
Let $H:X\times \R\to Y$ and $G:Z\times \R\to Y$ be tangential deformations (hence it is supposed that $X,Z\subset Y$, and $H(-,0)$ and $G(-,0)$ are inclusion maps) such that $H(X\times\{1\})\subset Z$. Then let $H*G$ be the the tangential deformation of $X$ defined by:
$$
H*G(x,t)=
\begin{cases}
H(x,2t) & \text{if $t\leq \frac{1}{2}$}\\
G(H(x,1),2t-1) & \text{if $t\geq\frac{1}{2}$\;.}
\end{cases}
$$
\end{defn}

Classical LS category is upper bounded by the dimension of the given space plus one. The classical result due to W. Singhof and E. Vogt \cite{Vogt-Singhof} states that, for any $C^2$ lamination \FF\ on a compact manifold, $\Cat(\FF)\leq\dim\FF+1$. We shall give an adaptation of these results to \LB-category. The vanishing result of \LB-category in minimal laminations will be a corollary of that.

\begin{prop}\label{p:homotopyholonomy}
Let $U$ be a tangentially categorical set, $T$ a complete
topological transversal, $H$ a tangential contraction for $U$ and
$\varepsilon>0$. Suppose that $\LB$ is finite in compact sets. Then there exists
two open sets $V,W$ covering $U$ and
tangential contractions $H^V,H^W$ of each $V$ and $W$ respectively such that
$$\LB(H^V(V\times\{1\}))+\LB(H^W(W\times\{1\}))\leq \LB(T) +
\varepsilon\;$$
\end{prop}

\begin{proof}
Observe that $T_H=H(U\times\{1\})$ is a transverse set (by Lemma~\ref{l:vogt0}); in
fact, for any $x\in U$, there exists an $\FF_U$-saturated open neighborhood $U_x\subset U$ such that $H(U_x\times\{1\})$ is a transversal associated to some foliated chart. Hence, since $T$ is a
complete transversal we can consider that each $H(U_x\times\{1\})$ is inserted by a holonomy map into $T$. By using Lemma~\ref{l:disgregation}, there exists a partition $\{F,V\}$ of $U$ such that $F$ is closed in $U$ and null transverse, and $V$ is an open set so that there exist a tangential contraction $H^V:V\times\R\to \FF$ satisfying $H(V\times\{1\})\subset T$. Now, observe that $\LB(H(F\times\{1\})=0$. Since $\LB$ is externally regular on \sg-compact sets (by the finiteness on compact sets), there exists an open neighborhood $S$ of $H(F\times\{1\})$ in $T_H$ such that $\LB(S)<\varepsilon$. Let $W=H(-,1)^{-1}(S)$. Clearly, these $V$ and $W$ satisfy the stated conditions.
\end{proof}

\begin{thm}
Let $(M,\FF,\LB)$ be a $C^2$ foliated compact manifold with a transverse
invariant measure that is finite in compact sets. Let $T$ be a complete transversal of $\FF$.
Then
$$\Cat(\FF,\LB)\leq (\dm\FF+1)\cdot\LB(T)\;.$$
\end{thm}

\begin{proof}
Since $\Cat(\FF)\leq\dm\FF+1$, we can use Proposition~\ref{p:homotopyholonomy} with
$\dm\FF +1$ tangentially categorical open sets covering $M$. Then
\[
  \Cat(\FF,\LB)\leq (\dm\FF+1)\cdot(\LB(T)+\varepsilon)
  \]
for all $\varepsilon>0$.
\end{proof}

\begin{cor}\label{c:minimalfoliation}
Let $(X,\FF)$ a minimal foliated manifold and let \LB\ be a
transverse invariant measure of \FF\ that is regular without atoms. Then $\Cat(\FF,\LB)=0$.
\end{cor}

A similar resut can be stated for ergodic laminations and, in general, in any lamination where there exists complete transversals with arbitrarily small measure. Remark that the regularity condition in the measure is essential in the previous corollary. By using suspensions of the diffeomorphisms given in \cite{Kodama-Matsumoto} we can obtain minimal flows on the torus with positive \LB-category. 

\section{Secondary \LB-category}

In this section, we introduce a refinement of the definition of the \LB-category. The \LB-category is zero in many interesting cases, as follows from the dimensional upper bound (see Corollary~\ref{c:minimalfoliation}). Now, we focus on the rate of convergence to zero of the expression that defines the \LB-category; this rate will be called the {\em secondary \LB-category}.

Let $(X,\FF)$ be a lamination on a compact Polish space. For a fixed finite regular atlas \UU\ of \FF, recall that a {\em chain of plaques\/} of $\UU$ is a finite sequence of plaques of charts in $\UU$ such that each plaque meets the next one. The {\em length\/} of a chain of plaques is its number of elements plus one. For $x,y\in M$ in the same leaf, let $d_{\UU}(x,y)$ be the minimum length of a chain of plaques of $\UU$ such that the first and last ones contain $x$ and $y$, respectively. If $x$ and $y$ are points in different leaves, we set $d_{\UU}(x,y)=\infty$. This defines a function $d_{\UU}:M\times M\to\N\cup\{\infty\}$ that is symmetric and satisfies the triangle inequality (it is called a {\em coarse metric\/} in \cite{Hurder}). Moreover $d_{\UU}(x,y)<\infty$ if and only if $x$ and $y$ are in the same leaf, and $d_{\UU}(x,y)=1$ if and only if there is some plaque containing $x$ and $y$. For $x\in M$ and $X\subset M$, let $d_{\UU}(x,X)=\min_{y\in X}d_{\UU}(x,y)$. The {\em length\/} relative to \UU\ of a path $\sigma$ on the leaves, denoted by $\length(\sigma)$ (or more explicitly, $\length_{\UU}(\sigma)$), is the minimum length of a chain of plaques in \UU\ covering $\sigma$. When $\FF$ is $C^1$, we can fix a Riemannian metric $g$ on the leaves that varies continuously in the transverse direction. Then, if $\sigma$ is Lipschitz, its {\em length\/} $\length(\sigma)$ can be also given by $g$; in this case, the notation $\length_g(\sigma)$ can be also used. From the compactness of $X$, it follows that there is some $A\ge1$ and $B\ge0$ such that
  \[
    \frac{1}{A}\,\length_g(\sigma)-B\le\length_{\UU}(\sigma)\le A\,\length_g(\sigma)+B
  \]
for all $\sigma$. Because of this, both definitions of length will work equally well for our purposes. Then we preferably use the length given by a regular atlas since it is defined with more generality. The {\em length\/} of a tangential deformation is the supremum of the lengths of the induced leafwise paths; if the length is defined by a leafwise Riemannian metric, we assume that the deformation is Lipschitz.

Let $\LB$ be a transverse invariant measure of $\FF$ such that $\Cat(\FF,\LB)=0$. In the rest of this section, we assume that $\Lambda$ is regular and finite on compact sets.

\begin{defn}
Let $\Cat(\FF,\LB,r)$ be defined like the \LB-category by using only tangential deformations of length $\le r$. If we use a finite covering $\UU$ to define the length, then the more explicit notation $\Cat(\FF,\LB,r,\UU)$ can be used. If we use a leafwise Riemannian metric $g$ to define the length, then we consider only Lipschitz deformations in the definition of $\Cat(\FF,\LB,r)$, and the more explicit notation $\Cat(\FF,\LB,r,g)$ can be used.
\end{defn}

Observe that $\Cat(\FF,\LB,n,\UU)$ is a non-increasing sequence of positive numbers that converges to $\Cat(\FF,\LB)=0$.

\begin{rem}
It is also true for $\Cat(\FF,\LB,n)$ that a null transverse set is unessential for its computation. Any \sg-compact null transverse set can be covered by tangentially contractible open sets contained in charts such that each chart has a contraction of $\length_{\UU}\le 1$, and the sum of the measures of the final deformations are arbitrarily small (see Proposition~\ref{p:nullset}).
\end{rem}

\begin{defn}[Growth types]
Let $\mathcal{I}$ be the set of non-negative non-decreasing sequences.
$$
  \mathcal{I}=\{\,g:\N\to[0,\infty)\mid g(n)\leq g(n+1)\ \forall n\in\N\,\}\;.
$$
 A preorder ``$\preceq$'' on $\mathcal{I}$ is defined by setting $g\preceq h$ if $\exists B\in\N$ and $\exists A>0$ so that $g(n)\leq A\, h(Bn)$ $\forall n\in\N$.
 This preorder $\preceq$ induces an equivalence relation ``$\simeq$'' in $\mathcal{I}$ defined by $g\simeq h$ if $g\preceq h\preceq g$. The elements of the quotient set $\mathcal{E}=\mathcal{I}/\simeq$ are called {\em growth types of non-decreasing sequences\/} and has the partial order ``$\leq$'' induced by ``$\preceq$''. The equivalence class of $g$ is denoted by $[g]$.
\end{defn}

Notice that the sequence $1/\Cat(\FF,\LB,n,\UU)$ is non-decreasing.

\begin{defn}
The growth type of  $1/\Cat(\FF,\LB,n,\UU)$ is called the {\em secondary \LB-category\/} of \FF, and is denoted by $\Cat_{\second}(\FF,\LB,\UU)$.
\end{defn}

\begin{prop}\label{p:Cat_II is independent of UU}
The secondary \LB-category is independent on the choice of $\UU$.
\end{prop}

\begin{proof}
Let $\VV$ be another finite regular atlas of $\FF$. There exists some integer $C\ge1$ such that each chart of $\UU$ is covered by at most $C$ charts of $\VV$. Hence any chain of charts of $\UU$ of length $n\in\N$ can be covered by a chain of charts of $\VV$ of length $\leq Cn$, giving $\length_{\UU}(\sigma)\le C\,\length_{\VV}(\sigma)$ for any leafwise path $\sigma$. Finally, any tangential homotopy follows, locally, a chain of charts \cite{Vogt-Singhof}. Hence any tangential homotopy of $\length_{\UU}\le n$ is of $\length_{\VV}\leq Cn$, obtaining $\Cat(\FF,\LB,Cn,\VV)\leq\Cat(\FF,\LB,n,\UU)$ for all $n$, and therefore $\Cat_{\second}(\FF,\LB,\UU)\leq\Cat_{\second}(\FF,\LB,\VV)$. The reverse inequality follows with the same argument.
\end{proof}

According to Proposition~\ref{p:Cat_II is independent of UU}, the secondary \LB-category is denoted by $\Cat_{\second}(\FF,\LB)$ from now on. Since the ambient space is compact, the distorsions in the leafwise length of homotopies given by a tangential homotopy equivalence are uniformly bounded. Therefore the growth types involving the secondary \LB-categories are the same. By using the same argument to show that the primary \LB-category is invariant by measure preserving tangential homotopy equivalences \cite{Menino1}, we obtain the following.

\begin{prop}\cite{Menino}
The secondary \LB-category is invariant by measure preserving tangential homotopy equivalences.
\end{prop}

Let $T$ be a complete transversal of $\FF$. Define $\Cat(\FF,\LB,n,\UU,T)$ in the same way as $\Cat(\FF,\LB,n,\UU)$ by taking only tangential contractions $H:U\times [0,1]\to\FF$ such that $H(U\times\{1\})\subset T$.

\begin{prop}\label{p:restricttotransversal}
 The growth types of
   \[
     \frac{1}{\Cat(\FF,\LB,n,\UU)}\quad\text{and}\quad\frac{1}{\Cat(\FF,\LB,n,\UU,T)}
   \]
 are equal.
\end{prop}

\begin{proof}
By using the argument of Proposition~\ref{p:homotopyholonomy}, any tangentially contractible open set $U$ has a partition $V_1,\dots,V_k,F$, where the sets $V_i$ are open, $\bigcup_i V_i$ contracts to a transversal contained in $T$, and $F$ is a null-transverse set. By compactness, since $T$ is complete, the length of a deformation of $\bigcup_i V_i$ needed to reach $T$ is bounded by a constant independent of $U$: the ambient space can be covered by a sequence of saturations of $T$ in the charts of a regular foliated atlas; by compactness, there is a finite number of them, and this number is the desired constant. By invariance, the measure of the final step of this deformation is controlled by the measure of a contraction of $U$. On the other hand the contribution of $F$ can be made as small as desired by using contractions of length $1$ (in the charts of the atlas).
\end{proof}

\section{Secondary category and growth of the holonomy pseudogroup}

Now, we give a relation between secondary \LB-category and the growth type of the holonomy group in the case of free suspensions. This result shows the deep relation between the concepts of holonomy and tangential homotopy. The growth rate of the group give us a clear lower bound for the secondary \LB-category. The converse is true with suitable conditions given by the Rohlin tower theorem.

\begin{defn}[Growth of a group]
Let $G$ be a finitely generated group and let $S$ be a symmetric set of generators (here, ``symmetric'' means that the identity element and the inverse of any element of $S$ belong to $S$). For each $g\in G$, let $l_S(g)$ be the minimum number $n\in\N$ such that $g$ can be expressed by the composition of at most $n$ elements of $S$. Set $S_n=\{\,g\in G\mid l_S(g)\leq n\,\}$. The {\em growth\/} of $G$ is the growth type of the sequence $\# S_n$, $\growth(G)=[\# S_n]$.
\end{defn}

The growth of a finitely generated group is independent of the choice of the finite set of generators.

We give a version of the notion of growth for pseudogroups as follows. In the case of pseudogroups, we have the operations of composition, inversion, restriction to open sets and combination. A set of generators is a set of elements of the pseudogroup such that any other transformation can be obtained from them by using the above operations. We assume that the set is symmetric in the sense that it contains the identity maps on the domains and images of its elements, and is closed by inversion; thus inversion can be removed from the above operations.

\begin{defn}[Growth of a pseudogroup \cite{Walczak}]
 Suppose that $\Gamma$ is finitely generated. Choose a finite symmetric set of generators $S$ and let $S_n$ be the set of transformations obtained by a composition of at most $n$ elements of $S$. Finally we define $\growth_S(\Gamma)=[\#S_n]$.

Let $\LB$ be a probability measure on $T$ invariant by $\Gamma$. Let $K\subset T$ be the support of $\LB$, which is a $\Gamma$-invariant closed subset of $T$. Then $\Gamma$ induces a pseudogroup $\Gamma_\LB$ on $K$, and a symmetric set $S$ of generators of $\Gamma$ induces a symmetric set of generators $S^\LB$ of $\Gamma_\LB$. The {\em \LB-growth\/} of $\Gamma$ associated to $S$ is $\growth_{S,\LB}(\Gamma)=\growth_{S^{\LB}}(\Gamma_\LB)$.
\end{defn}

The definition of growth of pseudogroups depends on the choice of the set of generators.

\begin{prop}[Upper bound for the secondary \LB-category]\label{p:upperboundsecondary}
 $\Cat_{\second}(\FF,\LB)\leq \growth_{S^\UU,\LB}(\Gamma)$, where $S^\UU$ is the finite symmetric system of generators defined by a regular foliated atlas $\UU$.
\end{prop}

\begin{proof}
Since $\Cat(\FF,\LB)=0$, we can suppose that $\LB(T)<\infty$, or even $\LB(T)=1$, where $T$ is a disjoint union of transversals associated to charts of a regular finite foliated atlas \UU. In order to compute the secondary $\LB$-category, we use $\UU$ to define the length of tangential homotopies, and Proposition~\ref{p:restricttotransversal} to consider only tangential contractions finishing in $T$.

A finite family of holonomy transformations can be associated to any tangential contraction of finite length since, locally, a tangential contraction follows a chain of charts \cite{Vogt-Sinhof}. For each $n\in\N$, let $H:U\times [0,1]\to\FF$ be a tangential contraction such that $U\cap \supp_T(\LB)\neq\emptyset$ and $\length(H)\leq n$. It is clear that each holonomy transformation $h$ locally defined by $H$ on $\supp_T(\LB)$ is a restriction of elements in $S^{\UU,\LB}_{n}$. Given any $\varepsilon>0$, let $U_1,\dots,U_N$ be tangentially categorical open sets covering the ambient space, and let $H^1,\dots,H^N$ be respective tangential contractions of length $\le n$ such that
$$
\sum_i\LB(H^i(U_i\times\{1\}))<\Cat(\FF,\LB,n,\UU,T) + \varepsilon\;.
$$
Hence $\{U_i\cap\supp_T(\LB)\}_{i=1}^N$ is a covering of $\supp_T(\LB)$, and each $H^i$ defines a finite number of local holonomy maps obtained from $S^{\UU,\LB}_n$ by restriction. Let $\{h_1,\dots,h_K\}$ be the set of such holonomy maps, and let $S^{\UU,\LB}_n=\{g_1,\dots,g_M\}$, where $M=\#S^\LB_n$. We can suppose that each $U_i$ has a partition $\{U_{i,j},F_i\}_{1\leq j\leq N_i}$, where $F_i$ is a null transverse set, and the sets $U_{i,j}$ are open and so that each $U_{i,j}\cap \supp_T(\LB)$ is contained in the dominion of some $h_k$ (see Proposition~\ref{p:homotopyholonomy}). For $1\leq i\leq M$, let $F_i$ denote the family of maps $h_j$ obtained from $g_i$ by restriction. All of the maps in $F_i$ can be combined  to give another holonomy transformation $q_i$ which is a restriction of $g_i$; thus all the holonomy maps in $F_i$ are restrictions of $q_i$. Since the domains of the maps $q_i$ cover $\supp_T(\LB)$, it follows that there is some $i_0\in\{1,\dots,M\}$ such that $\LB(\dom(q_{i_0}))\ge1/M$. By the invariance of the measure,
$$
\frac{1}{M}\leq\LB(\Image(q_{i_0}))\leq\sum_i\LB(H^i(U_i\times\{1\})<\Cat(\FF,\LB,n,\UU,T) + \varepsilon\;.
$$
The proof follows by taking $\varepsilon>0$ arbitrarily small.
\end{proof}

\begin{rem}\label{l:growthgroup}
In the case of foliations given by suspensions of free actions of finitely generated groups on a locally compact Polish space, we can take the bound given by the growth of the group.
\end{rem}

\section{Case of free suspensions with Rohlin groups}\label{s:Rohlin}

In this section, we see a family of examples where the inequality of Proposition~\ref{p:upperboundsecondary} becomes an equality. They are the cases where the pseudogroup is generated by a free action of a Rohlin group.

\begin{defn}[Rohlin towers, Rohlin sets and Rohlin groups]
Let $G$ be a locally compact, second countable and Hausdorff topological group acting on a standard Borel space $X$, let $\LB$ be an invariant probability measure on $X$, and let $F\subset G$ be a Borel subset. We say that a Borel set $V\subset X$ is an {\em F-base\/} in $X$ if $FV=\bigcup_{f\in F} f(V)$ is Borel, $\LB(FV)>0$, and the sets $f(V)$ ($f\in F$) are disjoint from each other. The set $FV$ in the previous definition is called an {\em $F$-tower\/}.

A relatively compact set $F\subset G$ is called a {\em Rohlin set\/} if, for any free action of $G$ on a standard Borel space $X$ with an invariant probability measure $\LB$ and for any $\varepsilon>0$, there exists an $F$-tower $FV\subset X$ with $\LB(FV)>1-\varepsilon$.

A topological group $G$ as above is called a {\em Rohlin group\/} if, for any compact subset $K\subset G$, there exists a Rohlin set $F\subset G$ so that $K\subset F$.
\end{defn}

\begin{thm}[C.~Series \cite{Series}]\label{t:series}
Any locally compact, Hausdorff, second countable, almost connected and amenable group is a Rohlin group.
\end{thm}

\begin{prop}[Open approximation to a Rohlin tower]\label{p:open approximation to a Rohlin tower}
Let $G$ be a discrete Rohlin group acting freely on a locally compact Polish space $T$ with a regular probability invariant measure $\LB$. Let $F$ be a Rohlin set, and let $\varepsilon>0$. Then there exists a sequence $\{V_k\}_{k\in\N}$ of open $F$-bases such that $\sum_k\LB(FV_k)> 1 - \varepsilon$.
\end{prop}

\begin{proof}
Any Rohlin set $F$ is finite since $G$ is discrete and $F$ is relatively compact by definition. Let $B$ be an $F$-base. By regularity of $\LB$, there exists an open set $V\supset B$ such that $|\LB(V)-\LB(B)|<\varepsilon/\# F$. By continuity, finiteness of $F$ and since the action is free, there is a partition $\{C,V_1,V_2,\dots\}$ of $V$ such that $C$ is closed on $V$, each $V_i$ is an open $F$-base and $\LB(C)=0$ ($C$ is the union of the boundaries of the sets $V_i$. Finally, by the invariance of $\LB$,
$$
  1-\varepsilon<\LB(FB)\leq \sum_k\LB(FV_k)\leq \# F\cdot\left(\LB(B)+ \frac{\varepsilon}{\# F}\right)\leq 1\;.\qed
$$
\renewcommand{\qed}{}
\end{proof}

\begin{prop} \label{p:IgualdadRohlin}
Let $M$ and $F$ be compact manifolds, and let $h:\pi_1(M)\to \Homeo(F)$ be a homomorphism. Suppose that $h(\pi_1(M))$ is a Rohlin group, the induced action of $\pi_1(M)$ on $F$ is free, and there exists an invariant regular probability measure $\LB$ on $F$. Then $\Cat_{\second}(\widetilde{M}\times_h F,\LB)=\growth(h(\pi_1(M)))$.
\end{prop}

\begin{proof}
The inequality $\Cat_{\second}(\widetilde{M}\times_h F,\LB)\leq\growth(h(\pi_1(M)))$ holds by Proposition~\ref{p:upperboundsecondary}. Now, let $U_1,\dots,U_M$ be a covering of $M$ by contractible open sets and let $p:\widetilde{M}\times_h F \to M$ be the projection from the suspension to the base space. We have $p^{-1}(U_i)\approx U_i\times F$  for all $i$, obtaining a regular foliated atlas $\UU$ for the computation of the secondary \LB-category. Let $\varepsilon>0$ and let $S$ be a symmetric set of generators of $h(\pi_1(M))$. By the Proposition~\ref{p:open approximation to a Rohlin tower}, for each $n\in\N$, there exists a sequence $V^n_k$ of mutually disjoint open $S_n$-bases (in $F$) such that $1\geq\sum_k\LB(S_n V^n_k)> 1 - \varepsilon$. Moreover $\LB(\bigcup_{k} V^n_k)\leq 1/\# S_n$ by the invariance of $\LB$.
By the regularity of $\LB$, there is a compact subset $K_n\subset \bigcup_k S_n V^n_k$ such that $1-2\varepsilon<\LB(K_n)<1-\varepsilon$; clearly, $\LB(F\setminus K_n)<2\varepsilon$. For each $k$, consider the saturation $\str_i(S_n V^n_k)$ of $S_n V^n_k$ on each $p^{-1}(U_i)\approx U_i\times F$; these sets are open in the ambient space and they are clearly tangentially contractible. The contractibility of $U_i$ induces a tangential contraction of $\str_i(S_n V^n_k)$ to the transversal $S_n V^n_k$, which in turn contracts to $V_k^n$ by a homotopy induced by the holonomy maps in $S_n$. Thus, by compactness, there exists a constant $K$ independent of $n$ such that all of these induced homotopies have length $\le Kn$. On the other hand, the saturations $\str_i(F\setminus K)$ in each $p^{-1}(U_i)$ contract to $F\setminus K$.

Therefore, for each $n\in\N$, the family
  \[
    \{\,\str_i(F\setminus K),\str_i(S_n V^n_k)\mid i\in\{1,\dots,M\},\ k\in\N\,\}
  \]
is a covering of $\widetilde{M}\times_h F$ by tangentially contractible open sets, giving
$$
\Cat(\FF,\LB,\UU,Kn)\leq 4M\varepsilon + \LB\left(\bigcup_k V_k\right)\leq 4M\varepsilon + \frac{M}{\# S_n}\;.\qed
$$
\renewcommand{\qed}{}
\end{proof}

\begin{cor}
Let $\FF$ be a lamination obtained by a free suspension of a group of homeomorphisms $G$ of a Polish space preserving a probability invariant measure. If $\Cat_{\second}(\FF,\LB)\neq\growth(G)$ then $G$ is not a Rohlin group.
\end{cor}

\begin{exmp}\label{exmp:secondary LB-category of a minimal foliation by hyperplanes on the n-dimensional torus}
The secondary \LB-category of a minimal foliation by hyperplanes on the $n$-dimensional torus with the Lebesgue measure in a transverse circle is $[1,2^n,3^n,\dots]$. Thus the secondary \LB-category distinguishes the dimension of the ambient manifold of these foliations.
\end{exmp}

\begin{rem}
The upper bound of Proposition~\ref{p:upperboundsecondary} is an equality in the case of Rohlin suspensions (Proposition~\ref{p:IgualdadRohlin}), but not in general as shown by the following simple example. Let $(T^3,\FF_1)$ be the foliation by dense cylinders and $(T^3,\FF_2)$ a minimal foliation by hyperplanes. Consider the usual Lebesgue invariant measures in a transverse circle in both cases. Let $\FF=\FF_1\sqcup \FF_2$ on $T^3\sqcup T^3$. We have $\Cat(\FF,\LB,k)\sim \frac{1}{k}+\frac{1}{k^2}$ and $\Cat_{\second}(\FF,\LB)=[1,2,3,\dots]$. However, $\growth_{\LB}(\Gamma_{\FF_1\sqcup\FF_2})=\growth_{\LB}(\Gamma_{\FF_2})=[1,2^2,3^2,\dots]$.
\end{rem}

\section{Pseudogroup invariance}

The nullity or the positivity of the \LB-category is an invariant of the holonomy pseudogroup and the invariant measure \cite{Menino}, and the secondary \LB-category is defined in the case of zero \LB-category. The aim of this section is to show the same kind of invariance for the secondary category. Consider the definitions and notation of Section~\ref{s:invariance}.

We recall here a result of topological dimension that will be useful in this section.

\begin{prop}[Dimensional trick \cite{James}]\label{p:dimensional trick}
Let $X$ be a paracompact space of finite topological dimension and let $\UU$ be an open covering of $X$. Then there exists a covering $\{V_0,\dots,V_{\dim X}\}$ of $X$ such that each $V_i$ is a union of a countable family of disjoint open sets $V_{ij}$ and the open covering $\{V_{ij}\}$ is a refinement of $\UU$.
\end{prop}

\begin{defn}
Let $\Gamma$ be a finitely generated pseudogroup of local transformations of a locally compact Polish space $T$. Let $\LB$ be an invariant measure of $\Gamma$ and suppose that $\Cat(\Gamma,\LB)=0$. Let $S$ be a symmetric set of generators of $\Gamma$. The {\em length\/} of a map $h\in\Gamma$ is the minimum integer $k$ such that $h$ can be locally expressed as a composition of at most $k$ elements in $S$. 

A {\em deformation\/} of $U\subset T$ is a map $h:U\to T$ that is combination of maps $h_i:U_i\to T$ in $S_\infty=\bigcup_n S_n$ restricted to disjoint open sets; we may use the notation $h\equiv(h_i)$. It is said that $U=\sqcup_i U_i$ is {\em deformable\/} if there is a deformation of $U$. The pairs $(U_i,h_i)$ are called {\em components\/} of $h$. Thus, if $h\equiv(h_i)$ is a deformation of an open set $U\subset T$ and each $h_i$ belongs to $S_k$, then we say that the {\em length\/} of the deformation is $\le k$. 
\end{defn}

\begin{defn}[Secondary \LB-category of pseudogroups]
With the above hypothesis and notations, let
$$
\Cat(\Gamma,\LB,S,n)=\inf_{\UU,h^U}\sum_{U\in\UU}\Lambda(h^U(U))\;,
$$
where $\UU$ runs in the family of open coverings of $T$, and, for each $U\in\UU$, $h^U\equiv(h^U_i)$ runs in the family of deformations of length $\leq n$ of $U$. The {\em secondary \LB-category\/} of $(\Gamma,S)$ is $\Cat_{\second}(\Gamma,\LB,S)=[1/\Cat(\Gamma,\LB,S,n)]$.
\end{defn}

\begin{defn}[Compact generation \cite{Haefliger}]\label{d:compact generation}
Let $\Gamma$ be a pseudogroup of local transformations of a
locally compact space $T$. It is said that $\Gamma$ is {\em compactly generated\/} if there is a relatively compact
open set $U$ in $T$ meeting each orbit of $\Gamma$, and such that the restriction $\mathcal{H}$ of $\Gamma$ to $U$ is
generated by a finite symmetric collection $S\subset \mathcal{H}$ so that each $g\in S$ is the restriction
of an element $\overline{g}$ of $\Gamma$ defined on some neighborhood of the closure of $\dom(g)$. The set $S$ is called a system of compact generation of $\Gamma$ in $U$.
\end{defn}

The holonomy pseudogroup of a lamination on a compact space is compactly generated.

\begin{defn}[\cite{Alvarez-Candel}]\label{d:recurrent system}
A finite symmetric family $E$ of generators of a pseudogroup $\Gamma$ of local
transformations of a locally compact space $T$ is said to be {\em recurrent\/} if there exists a
relatively compact open subset $V\subset T$ and some $R > 0$ such that, for any $x\in Z$, there exists $h\in \Gamma$ with $x\in\dom h$, $\length_E(h)<R$ and $h(x)\in V$.
\end{defn}

\begin{defn}\label{recurrent compactly generated pseudogroup}
According to the notation in Definitions~\ref{d:compact generation} and \ref{d:recurrent system}, a pseudogroup $\Gamma$ is called a recurrent compactly generated pseudogroup if $\Gamma$ is compactly generated and recurrent, and there exists a recurrent system $E$ of $\Gamma$ such that $V\subset U$, where $U$ and $V$ are like in Definitions~\ref{d:compact generation} and~\ref{d:recurrent system}, respectively.
\end{defn}

\begin{prop}\label{p:Cat_second(Gamma,LB,E)=Cat_second(HH,LB,S)}
Let $\Gamma$ be a recurrent compactly generated pseudogroup in a locally compact Polish space $T$ of finite dimension, let $\LB$ be a $\Gamma$-invariant measure, let $S$ be a system of compact generation in $U$, let $\HH$ be the restriction of $\Gamma$ to $U$ and let $E$ be a recurrent system of generators of $\Gamma$ on $V\subset U$. Then $\Cat_{\second}(\Gamma,\LB,E)=\Cat_{\second}(\HH,\LB,S)$.
\end{prop}

\begin{proof}
Since $E$ is recurrent and using the trick of Proposition~\ref{p:homotopyholonomy}, we can suppose that we have a covering $\{U_1,\dots,U_N\}$ of $T$ and deformations $h_i:U_i\to V$ with $\length_E(h_i)\le R$. Since $U$ is relatively compact, and any $g\in S$ extends to a map $\overline{g}\in\Gamma$ defined in a neighborhood of $\overline{U}$, it follows that any map $g\in S$ is a combination of maps in $E_K$ for $K\in\N$ large enough. Using the decomposition trick once more, we obtain that any deformation of $\length_S\le n$ induces, up to a null transverse set, a deformation of length $\le Kn$ relative to $E$ for $K\in\N$ large enough. We can suppose, as usual, that $N=\dim T + 1$, and therefore
\[
  \Cat(\Gamma,\LB,Kn + R,E)\leq (\dim T + 1)\cdot\Cat(\HH,\LB,S,n)\;,
\]
obtaining $\Cat_{\second}(\HH,\LB,S)\leq\Cat_{\second}(\Gamma,\LB,E)$.

The reverse inequality is easier by using the recurrency in $E$. There exists a constant $K'$ such that the composition $f\circ h\circ g^{-1}$ of any $h\in E$ with two returning maps, $f:\Image h \to V\subset U$, $g:\dom_h\to V\subset U$, of length $\le R$ relative to $E$ is of length $\le K'$ relative to $S$ for $K'$ large enough. Therefore $\Cat(\HH,\LB,K'n,S)\leq \Cat(\Gamma,\LB,n,E)$ since, locally,
\[
  h_{i_n}\circ\dots\circ h_{i_1}=f_{i_n}\circ h_{i_n}\circ g^{-1}_{i_n}\circ\dots\circ f_{i_1}\circ h_{i_1}\circ g^{-1}_{i_1}\;,
\]
where $f_{i_j}$ and $g_{i_j}$ are returning maps of length $\le R$ relative to $E$.
\end{proof}

\begin{cor}\label{c:independence in the set of generators}
The secondary \LB-category of recurrent compactly generated pseudogroups is independent of the choice of the recurrent set of generators.
\end{cor}

\begin{proof}
If $E$ and $E'$ are two different recurrent systems of generators for $\Gamma$ relative to open subsets $V\subset U$ and $V'\subset U'$, respectively, where $U$ and $U'$ are relative compact sets associated to the systems of compact generation $S$ and $S'$, respectively, and let $\HH$ and $\HH'$ be the restrictions of $\Gamma$ to $U$ and $U'$, respectively. Then $E''=E\cup E'$ is another recurrent system compatible with $V\subset U$ and $V'\subset U'$.
Therefore, by using Proposition~\ref{p:Cat_second(Gamma,LB,E)=Cat_second(HH,LB,S)}, we obtain $\Cat_{\second}(\HH,\LB,S)=\Cat_{\second}(\Gamma,\LB,E'')=\Cat_{\second}(\HH',\LB,S')$. By the same argument $\Cat_{\second}(\Gamma,\LB,E)=\Cat_{\second}(\Gamma,\LB,E')$.
\end{proof}

\begin{rem}
According to Corollary~\ref{c:independence in the set of generators}, we can remove the recurrent set of generators in the notation of the secondary \LB-category of pseudogroups, using simply $\Cat_{\second}(\Gamma,\LB)$.
\end{rem}

\begin{prop}
$\Cat_{\second}(\Gamma,\LB)$ is an invariant by measure preserving equivalences between recurrent compactly generated pseudogroups with invariant measures.
\end{prop}

\begin{proof}
Let $\Phi$ a measure preserving equivalence between recurrent compactly generated pseudogroups with invariant measures, $(\Gamma,T,\LB)$ and $(\Gamma',T',\LB')$. The pseudogroup $\Gamma_{\Phi}$ on $T_\Phi=T\sqcup T'$ generated by $\Phi$, $\Gamma$ and $\Gamma'$ is recurrent and compactly generated, and the combination of the measures $\LB$ and $\LB'$, $\LB_\Phi$ on $T_\Phi$, is $\Gamma_\Phi$-invariant. So $\Cat_{\second}(\Gamma,\LB)=\Cat_{\second}(\Gamma_\Phi,\LB_\Phi)=\Cat_{\second}(\Gamma',\LB')$ by Proposition~\ref{p:Cat_second(Gamma,LB,E)=Cat_second(HH,LB,S)}.
\end{proof}

\begin{rem}
The holonomy pseudogroup given by a regular foliated atlas of a lamination in a compact space is recurrent and compactly generated. Observe that the coarse quasi-isometry type of the orbits is independent of the choice of a recurrent system of compact generation, which was the first reason to introduce the concept of recurrency on compactly generated pseudogroups \cite{Alvarez-Candel}.
\end{rem}

\begin{prop}
Let $(X,\FF,\LB)$ be a lamination on a compact space with a transverse invariant measure such that $\Cat(\FF,\LB)=0$. Let $\Gamma$ be the holonomy pseudogroup of $\FF$ on a complete transversal $T$. Then $\Cat_{\second}(\FF,\LB)=\Cat_{\second}(\Gamma,\LB)$.
\end{prop}

\begin{proof}
Let $\UU=\{U_1,\dots,U_M\}$ be a regular foliated atlas of $\FF$. The inequality $\Cat_{\second}(\FF,\LB)\leq \Cat_{\second}(\Gamma,\LB)$ is obvious since $\Cat(\Gamma,\LB,S^\UU,n)\leq\Cat(\FF,\LB,\UU,n,T)$, clearly. By using the dimensional trick (see Proposition~\ref{p:dimensional trick}),  we can show that $\Cat(\FF,\LB,\UU,n)\leq M(m+1)\Cat(\Gamma,\LB,S^\UU,n)$, where $m$ is the dimension of $X$. Let $\{V_i\}_{i\in\N}$ be an open covering of $T$, a complete transversal associated to the atlas $\UU$. Let $\str_{i}(B)$ be the saturation of a set $B$ relative to the chart $U_i\in\UU$. Of course, $\{\str_j(V_i\cap T_j)\}_{i,j}$ is an open covering of $X$. Now let $D_1,\dots,D_{\dim X + 1}$ be an open covering of $X$  such that $D_i=\bigsqcup_j D_{ij}$, where the $D_{ij}$ are open and mutually disjoint, and the whole collection $\{D_{ij}\}_{i,j}$ is a refinement of $\{\str_j(V_i\cap T_j)\}_{i,j}$ (by Proposition~\ref{p:dimensional trick}). Let $(k,l)=I(D_{ij})$ be the first indexes relative to the lexicographic order such that $D_{ij}\subset \str_l(V_k\cap T_l)$. Let $W_{ikl}$ be the union of sets $D_{ij}$ such that $I(D_{ij})=(k,l)$. Each set $W_{ikl}$ contracts to $V_k$ by a deformation of length $0$ (the contraction of the chart $U_l$ to its transversal) and therefore any deformation $h_i$ of length $n$ of $V_i$ induces a tangential deformation $H$ of each $W_{ikl}$ of length $n$ such that $H(W_{ikl}\times\{1\})\subset h(V_i)$. Hence
\[
  \sum_{ikl}\LB(H(W_{ikl}\times\{1\}))\leq M(m+1)\sum_i\LB(h_i(V_i))\;.\qed
\]
\renewcommand{\qed}{}
\end{proof}

\begin{cor}
The secondary \LB-category of compact laminations with transverse invariant measures is a transverse invariant {\rm(}it only depends on the recurrent compactly generated representatives of the holonomy pseudogroup with the corresponding invariant measure{\rm)}.
\end{cor}

\section{Final comments}

Many of the properties of tangential category and \LB-category can be adapted to secondary \LB-category, like the upper semicontinuity \cite{Vogt-Singhof} or the estimation by using the measure of the leafwise critical points \cite{Menino,Menino3}.

Considering its relation with the growth of groups, secondary \LB-category is an intermediate invariant which can interconnect analitical properties (and variational problems)  with (pseudo)group theory. 

It can be used in both directions. The secondary \LB-category could  provide interesting results for the non-Rohlin dynamics.

\begin{ack}
This paper contains part of my PhD thesis, whose advisor is Prof.
Jes\'{u}s A. \'{A}lvarez L\'{o}pez.
\end{ack}

\end{document}